\documentclass[letterpaper, 10pt, conference]{ieeeconf}
\IEEEoverridecommandlockouts
\usepackage{cite}

\usepackage{amssymb,amsfonts,amsmath,amsthm}
\usepackage{mathtools}
\usepackage{enumerate}
\usepackage{verbatim}
\usepackage{commath}
\usepackage[mathcal]{euscript}
\usepackage[usenames]{color}

\definecolor{plum} {rgb}{.4,0,.4}
\definecolor{BrickRed} {rgb}{0.6,0,0}

\usepackage[plainpages=false,pdfpagelabels,colorlinks=true,linkcolor=BrickRed,citecolor=plum]{hyperref}

\def\ddefloop#1{\ifx\ddefloop#1\else\ddef{#1}\expandafter\ddefloop\fi}

\def\ddef#1{\expandafter\def\csname b#1\endcsname{\ensuremath{\boldsymbol{#1}}}}
\ddefloop ABCDEFGHIJKLMNOPQRSTUVWXYZabcdefghijklmnopqrstuvwxyz\ddefloop

\def\ddef#1{\expandafter\def\csname c#1\endcsname{\ensuremath{\mathcal{#1}}}}
\ddefloop ABCDEFGHIJKLMNOPQRSTUVWXYZ\ddefloop
 
\def\ddef#1{\expandafter\def\csname s#1\endcsname{\ensuremath{\mathsf{#1}}}}
\ddefloop ABCDEFGHIJKLMNOPQRSTUVWXYZ\ddefloop

\def\Reals{{\mathbb R}}
\def\Ex{{\mathbf E}} 
\def\Pr{{\mathbf P}} 
\def\trn{{\hbox{\it\tiny T}}} 
\def\eps{\varepsilon}
\def\deq{:=}

\DeclareMathOperator*{\argmin}{arg\,min}

\newtheorem{theorem}{Theorem}

\newtheorem{lemma}{Lemma}

\newtheorem{remark}{Remark}

\title{\LARGE{Variational Principles for Mirror Descent and Mirror Langevin Dynamics}}

\author{Belinda Tzen${}^*$, Anant Raj${}^{\dagger\S}$, Maxim Raginsky${}^\S$, and Francis Bach${}^{\dagger}$%
\thanks{This work was supported in part by the NSF under award CCF-2106358 (``Analysis and Geometry of Neural Dynamical Systems'') and in part by the Illinois Institute for Data Science and Dynamical Systems
(iDS${}^2$), an NSF HDR TRIPODS institute, under award CCF-1934986. A.~Raj was supported by the Marie Sklodowska-Curie Fellowship (project NN-OVEROPT 101030817).}%
\thanks{${}^*$Columbia University, New York, NY, USA.}%
\thanks{${}^\S$University of Illinois, Urbana, IL, USA.}%
\thanks{${}^\dagger$INRIA, Ecole Normale Sup\'erieure, PSL Research University, Paris, France.}%
\thanks{Emails: bt2314@columbia.edu, anant.raj@inria.fr, maxim@illinois.edu, francis.bach@inria.fr}
}

\begin{document}

\maketitle


\begin{abstract}
Mirror descent, introduced by Nemirovski and Yudin in the 1970s, is a primal-dual convex optimization method that can be tailored to the geometry of the optimization problem at hand through the choice of a strongly convex potential function. It arises as a basic primitive in a variety of applications, including large-scale optimization, machine learning, and control. This paper proposes a variational formulation of mirror descent and of its stochastic variant, mirror Langevin dynamics. The main idea, inspired by the classic work of Brezis and Ekeland on variational principles for gradient flows, is to show that mirror descent emerges as a closed-loop solution for a certain optimal control problem, and the Bellman value function is given by the Bregman divergence between the initial condition and the global minimizer of the objective function. 
   
\end{abstract}


\section{Introduction}

The continuous-time gradient flow
\begin{align}\label{eq:gradflow}
	\dot{x}(t) = - \nabla f(x(t)), \qquad x(0)=x_0
\end{align}
for a $C^1$ objective function $f : \Reals^n \to \Reals$ is a basic primitive in continuous-time optimization and control. Under additional assumptions on the objective $f$, the trajectory of \eqref{eq:gradflow} converges to a minimizer of $f$, which justifies thinking of the gradient flow as a method for asymptotically solving the optimization problem
\begin{align}\label{eq:min_f}
	\text{minimize } f(x), \quad x \in \Reals^n.
\end{align}
However, apart from the local characterization of $-\nabla f(x)$ as the ``direction of steepest descent,'' there appears to be little discussion of the sense, if any, in which \eqref{eq:gradflow} is ``optimal'' among all dynamical systems that asymptotically solve \eqref{eq:min_f}.

One of the few exceptions is the variational principle of Brezis and Ekeland \cite{Brezis_Ekeland,Ghoussoub_Tzou_gradflows}: Fix an arbitrary time horizon $T > 0$. Then, among all absolutely continuous curves $x : [0,T] \to \Reals^n$ with $x(0) = x_0$, the trajectory of \eqref{eq:gradflow} on $[0,T]$ minimizes the action functional
\begin{align}\label{eq:BE_action}
	S(x(\cdot)) \deq \int^T_0 \{ f(x(t)) + f^*(-\dot{x}(t)) \} \dif t + \frac{1}{2}|x(T)|^2,
\end{align}
where
\begin{align*}
	f^*(v) \deq \sup_{x \in \Reals^n} \{ \langle v,x \rangle - f(x) \}
\end{align*}
is the Legendre--Fenchel conjugate of $f$ and $|\cdot|$ denotes the Euclidean norm on $\Reals^n$. The minimum value of $S$ over all such $x(\cdot)$ is equal to $\frac{1}{2}|x_0|^2$. The underlying idea is simple and boils down to a careful analysis of the equality cases of the Fenchel--Young inequality
\begin{align*}
	f(x) + f^*(v) \ge \langle v,x \rangle.
\end{align*}
However, the Brezis--Ekeland variational principle does not say anything about the asymptotic behavior of the extremal trajectories or about the curious fact that a finite-horizon problem of minimizing the action \eqref{eq:BE_action} has a solution given by the flow of a time-invariant dynamical system.

In this paper, we revisit this problem from a control-theoretic point of view and provide a new variational interpretation of the gradient flow as an \textit{infinite-horizon stabilizing optimal control} \cite[\S8.5]{Sontag_system_theory}. Moreover, we consider a more general method of \textit{mirror descent}. This method, introduced in the 1970s by Nemirovski and Yudin \cite[Ch.~3]{Nemirovski_Yudn}, can be tailored to the geometry of the optimization problem at hand through the choice of a strongly convex \textit{potential function}. In continuous time, mirror descent is implemented by a time-invariant dynamical system whose state $x(t) \in \Reals^n$ and output $y(t) \in \Reals^n$ evolve according to
\begin{align}\label{eq:mirror_flow}
	\begin{split}
	\dot{x}(t) &= -\nabla f (\nabla \varphi^*(x(t))), \qquad x(0) = x_0 \\
	y(t) &= \nabla \varphi^*(x(t))
	\end{split}
\end{align}
where $\varphi : \Reals^n \to \Reals$ the potential function and $\varphi^*$ is its Legendre--Fenchel conjugate. It is common to refer to $x(t)$ and $y(t)$, respectively, as the \textit{dual-space} and the \textit{primal-space} trajectories. (The Euclidean gradient flow \eqref{eq:gradflow} is a special case of \eqref{eq:mirror_flow} with $\varphi(x) = \frac{1}{2}|x|^2$.)

\subsection{Brief summary of contributions}

We show the following: Suppose that the objective $f$ is strictly convex. For the controlled system $\dot{x}(t) = u(t), y(t) = \nabla\varphi^*(x(t))$, we consider the class of all \textit{stabilizing controls} \cite[\S8.5]{Sontag_system_theory}, i.e., appropriately well-behaved functions $u : [0,\infty) \to \Reals^n$ such that the output $y(t)$ converges, as $t \to \infty$, to the unique minimizer of $f$. This class contains, among others, sufficiently smooth state feedback controls of the form $u(t) = k(x(t))$. We then identify an instantaneous cost function $q(x,u)$, closely related to the Lagrangian in \eqref{eq:BE_action}, such that the state feedback law $k(x) = - \nabla f(\nabla\varphi^*(x))$ [or, equivalently, the output feedback law $\tilde{k}(y) = -\nabla f(y)$] gives a control that minimizes the infinite-horizon cost
\begin{align}\label{eq:infinite_horizon_cost}
	\int^\infty_0 q(x(t),u(t)) \dif t
\end{align}
over all stabilizing controls. Moreover, the value function $V(x_0)$, i.e., the minimum value of this cost as a function of the initial state $x_0$, is given by a certain ``distance,'' induced by the potential $\varphi$, between the initial output $y_0 = \nabla \varphi^*(x_0)$ and the unique minimizer of $f$. At the same time, $V(x)$ is the Lyapunov function for the closed-loop system $\dot{x}(t) = -\nabla f(\nabla \varphi^*(x(t)))$.

We also consider a stochastic variant of mirror descent, the so-called \textit{mirror Langevin dynamics} \cite{Hsieh_MLD,Chewi_MLD,Li_MLD}, which is implemented by an It\^o stochastic differential equation
\begin{align}\label{eq:MLD}
	\begin{split}
	\!\!\!\dif X_t &= - \nabla f (\nabla \varphi^*(X_t)) \dif t + \sqrt{2\eps (\nabla^2 \varphi^*(X_t))^{-1}} \dif W_t \\
	\!\!\! Y_t &= \nabla \varphi^*(X_t)
	\end{split}
\end{align}
driven by a standard $n$-dimensional Brownian motion $(W_t)_{t \ge 0}$, where $\eps > 0$ is a small ``temperature'' parameter. In this setting, we consider a finite-horizon optimal control problem of minimizing the expected cost
\begin{align*}
	\Ex\Bigg[\int^T_0 q(X_t,u_t)\dif t + r(X_T)\Bigg|X_0 = x_0\Bigg],
\end{align*}
over all admissible control processes $(u_t)_{0 \le t \le T}$ entering into the controlled SDE
\begin{align*}
	\dif X_t = u_t \dif t +  \sqrt{2\eps (\nabla^2 \varphi^*(X_t))^{-1}} \dif W_t.
\end{align*}
Here, $q$ is the same cost as in \eqref{eq:infinite_horizon_cost} and $r$ is an appropriately chosen terminal cost. As in the deterministic case, the mirror Langevin dynamics \eqref{eq:MLD} emerges as the closed-loop system corresponding to the optimal control, although the value function is now time-dependent. 

\section{The deterministic problem}

\subsection{Some preliminaries}
\label{ssec:prelims}

We assume that the objective function $f : \Reals^n \to \Reals$ is $C^1$ and strictly convex and that the potential function $\varphi : \Reals^n \to \Reals$ is $C^2$ and strictly convex. We let $\bar{y}$ denote the unique global minimizer of $f$. Both $f$ and $\varphi$ are assumed to be of \textit{Legendre type}, i.e., $|\nabla f(x)|,|\nabla \varphi(x)| \to + \infty$ as $|x| \to +\infty$. As a consequence (see, e.g., \cite[Thm.~26.5]{Rockafellar_book}), the gradient map $\nabla \varphi : \Reals^n \to \Reals^n$ is bijective, with $(\nabla \varphi)^{-1} = \nabla \varphi^*$. The potential $\varphi$ and its conjugate $\varphi^*$ induce the so-called \textit{Bregman divergences}
\begin{align}\label{eq:Bregman}
	\begin{split}
	D_\varphi(y,y') &\deq \varphi(y) - \varphi(y') - \langle \nabla\varphi(y'), y - y' \rangle \\
	D_{\varphi^*}(x,x') &\deq \varphi^*(x) - \varphi^*(x') - \langle \nabla\varphi^*(x'), x - x' \rangle
	\end{split}
\end{align}
and the following relation holds for any pair of points $y,y' \in \Reals^n$ and their ``mirror images'' $x = \nabla \varphi(y), x' = \nabla \varphi(y')$:
\begin{align*}
	D_\varphi(y,y') = D_{\varphi^*}(x',x)
\end{align*}
(cf.~\cite[\S11.2]{PLG} for details). We also require that, for each fixed $x' \in \Reals^n$, the map $x \mapsto D_{\varphi^*}(x,x')$ is \textit{radially unbounded}, i.e., $D_{\varphi^*}(x,x') \to + \infty$ as $|x| \to +\infty$. This will be satisfied, for example, if $\varphi$ is strongly convex, i.e., there exists some $\alpha > 0$, such that
\begin{align}\label{eq:phi_SC}
	\varphi(y') \ge \varphi(y) + \langle \nabla \varphi(y), y'-y \rangle + \frac{\alpha}{2}|y' - y|^2
\end{align}
for all $y,y' \in \Reals^n$. Radial unboundedness is needed for the invocation of the Lyapunov criterion for global asymptotic stability \cite[Sec.~5.7]{Sontag_system_theory}.
\begin{remark} {\em We assume that $\varphi$ is finite on all of $\Reals^n$ mainly to keep the exposition simple. It is not hard to adapt the analysis to the case when the potential function $\varphi$ is defined on a closed convex set $\sX \subseteq \Reals^n$ with nonempty interior, and $|\nabla \varphi(x)| \to + \infty$ as $x$ approaches any point on the boundary of $\sX$. This  corresponds to the problem of minimizing $f(x)$ subject to the constraint $x \in \sX$.}
\end{remark}

\subsection{Infinite-horizon optimal stabilizing controls}
\label{ssec:infinite_horizon}

Consider the time-invariant controlled dynamical system
\begin{align}\label{eq:basic_system}
	\dot{x}(t) = u(t).
\end{align}
For any $x_0 \in \Reals^n$, let $\cU_{x_0}$ denote the class of all \textit{stabilizing controls} at $x_0$, i.e., all locally essentially bounded maps $u : [0,\infty) \to \Reals^n$, such that the trajectory $x(t)$ of \eqref{eq:basic_system} with $x(0) = x_0$ is defined for all $t \ge 0$ and $x(t) \to \bar{x}$ as $t \to \infty$, where $\bar{x} \deq \nabla \varphi(\bar{y})$. We would like to minimize the cost
\begin{align*}
	J_\infty(x_0, u(\cdot)) \deq \int^\infty_0 q(x(t),u(t))\dif t,
\end{align*}
over all $u(\cdot) \in \cU_{x_0}$, where
\begin{align}\label{eq:BE_cost}
	q(x,u) \deq f(\nabla \varphi^*(x)) + f^*(-u) + \langle u, \bar{y} \rangle.
\end{align}
The class $\cU_{x_0}$ is evidently nonempty since, for example, the control $u(t) = (\bar{x}-x_0){\mathbf 1}_{\{0 \le t \le 1\}}$ is stabilizing at $x_0$. We denote by $V(x_0)$ the \textit{value function}, i.e., infimum of $J_\infty(x_0, u(\cdot))$ over all $u(\cdot) \in \cU_{x_0}$.

\subsection{The main result}

Let $V(x) \deq D_{\varphi^*}(x,\bar{x})$. Theorem~\ref{thm:MD_optimality}, stated and proved below, states that $V$ is the value function for the above infinite-horizon optimal control problem, and that  the mirror descent dynamics \eqref{eq:mirror_flow} is the closed-loop system corresponding to an optimal stabilizing control. Moreover, the value function $V$ is also a global Lyapunov function for \eqref{eq:mirror_flow}, and the point $\bar{x}$ is its global asymptotically stable equilibrium. The proof makes essential use of the following lemma, which we will also need at several points in the sequel:

\begin{lemma}\label{lm:Vdot_inequality} The function $V$ has the following properties:
	\begin{enumerate}
		\item It is $C^2$ and strictly convex.
		\item $V(\bar{x}) = 0$, and $V(x) > 0$ for $x \neq \bar{x}$.
		\item $V(x) \to +\infty$ as $|x| \to +\infty$.
	\end{enumerate}
Moreover, the following inequality holds for $\dot{V}(x,u) \deq \langle \nabla V(x), u \rangle$:
	\begin{align}\label{eq:Vdot_inequality}
		\dot{V}(x,u) + q(x,u) \ge 0, \qquad x,u \in \Reals^n
	\end{align}
	and equality is attained iff $u = -\nabla f (\nabla \varphi^*(x))$.
\end{lemma}
\begin{proof} Items 1)--3) are immediate consequences of our assumptions on $\varphi$. Moreover, a simple computation shows that
	\begin{align*}
		 \dot{V}(x,u) + q(x,u)  = f(\nabla \varphi^*(x)) + f^*(-u) + \langle u,\nabla \varphi^*(x) \rangle,
	\end{align*}
	which is nonnegative by the Fenchel--Young inequality. The equality condition in \eqref{eq:Vdot_inequality} follows by \cite[Thm.~23.5]{Rockafellar_book}.
\end{proof}

\begin{theorem}\label{thm:MD_optimality} We have the following:
	\begin{enumerate}
		\item For any stabilizing control $u(\cdot) \in \cU_{x_0}$ and for all $t \ge 0$,
		\begin{align}\label{eq:suboptimal_controls}
			\int^{t}_{0} q(x(t),u(t)) \dif t \ge V(x_0) - V(x(t)).
		\end{align}
		In particular, $J_\infty(x_0,u(\cdot)) \ge V(x_0)$.
		\item For each $x_0$, the closed-loop system $\dot{x}(t) = -\nabla f(\nabla\varphi^*(x(t)))$ gives rise to an optimal stabilizing control $u(t) = -\nabla f (\nabla \varphi^*(x(t)))$, such that
		\begin{align*}
			J_\infty(x_0, u(\cdot)) = V(x_0) = D_{\varphi}(x_0,\bar{x}).
		\end{align*}
		Moreover, $V(x)$ is a global Lyapunov function for the closed-loop system.
	\end{enumerate}
\end{theorem}
\begin{proof} Let $x_0 \in \Reals^n$ be given and consider an arbitrary stabilizing control $u(\cdot) \in \cU_{x_0}$. Then, for $\dot{x}(t) = u(t)$ with $x(0) = x_0$ we have
\begin{align*}
	V(x(t)) - V(x_0) &= \int^t_0 \frac{\dif}{\dif s}V(x(s)) \dif s \\
	&= \int^t_0 \dot{V}(x(s),u(s)) \dif s \\
	&\ge - \int^t_0 q(x(s),u(s)) \dif s,
\end{align*}
where the last step follows from \eqref{eq:Vdot_inequality}. Rearranging gives \eqref{eq:suboptimal_controls}. Moreover, taking the limit as $t \to \infty$ and using the fact that  $V(x(t)) \to 0$ as $t \to \infty$ since $u(\cdot)$ is stabilizing, we get the inequality $J_\infty(x_0, u(\cdot)) \ge V(x)$.

Next, consider the closed-loop system $\dot{x}(t) = -\nabla f (\nabla \varphi^*(x(t)))$, $x(0) = x_0$, that generates the mirror descent flow. Then $\bar{x}$ is evidently an equilibrium point since $\nabla f(\nabla\varphi^*(\bar{x})) = \nabla f(\bar{y}) = 0$. Letting $y(t) \deq \nabla \varphi^*(x(t))$ and using \eqref{eq:Vdot_inequality}, we have
\begin{align*}
	\frac{\dif}{\dif t}V(x(t)) &= \dot{V}(x(t),-\nabla f(y(t))) \nonumber\\
	&  = - q(x(t), - \nabla f(y(t))) \\
	& = \langle \nabla f(y(t)), \bar{y} \rangle - f^*(\nabla f(y(t))) - f(y(t)) \\
	& \le f(\bar{y}) - f(y(t)),
\end{align*}
which is strictly negative whenever $y(t) \neq \bar{y}$ by the strict convexity of $f$, or, equivalently, whenever $x(t) \neq \bar{x}$ since $\nabla \varphi^*$ is a bijection. Together with Lemma~\ref{lm:Vdot_inequality}, this shows that $V$ is a global Lyapunov function \cite[Def.~5.7.1]{Sontag_system_theory} for the above closed-loop system, so $\bar{x}$ is a globally asymptotically stable equilibrium \cite[Thm.~17]{Sontag_system_theory}. Thus, the control $u(t) = -\nabla f (\nabla \varphi^*(x(t)))$ is stabilizing at $x_0$, and $J_\infty(x_0, u(\cdot)) = V(x_0)$ from the equality condition in \eqref{eq:Vdot_inequality}.
\end{proof}

\subsection{Quantitative estimates}

Theorem~\ref{thm:MD_optimality} allows us to obtain quantitative estimates on the approach of the trajectory of \eqref{eq:mirror_flow} to equilibrium. While similar estimates have been given in some earlier works \cite{Krichene_AMD,Telgarsky_linopt}, the appeal of our optimal control perspective is that it allows to obtain such guarantees in a unified manner. It will be useful to introduce the following definition \cite{Bartlett_etal_AOGD}: We say that the objective function $f$ is \textit{$\mu$-strongly convex} ($\mu \ge 0$) w.r.t.\ the potential function $\varphi$ if
\begin{align*}
	f(y') \ge f(y) + \langle \nabla f(y), y'-y \rangle + \mu D_\varphi(y',y), \,\, y,y' \in \Reals^n.
\end{align*}
(Obviously, if $\mu = 0$, this is simply convexity; when $\mu > 0$, the function $f$ has some nonzero ``curvature'' in some neighborhood of each point $x$, where the ``geometry'' is determined by the potential $\varphi$.)

\begin{theorem}\label{thm:MD_quantitative} Let $(x(t),y(t))$, $t \ge 0$, be the state and the output trajectories of the mirror descent dynamics \eqref{eq:mirror_flow} starting from $x(0) = x_0$ and $y(0) = y_0 = \nabla \varphi^*(x_0)$. Then the following holds for every $t > 0$:
	\begin{enumerate}
		\item If $f$ is convex, then
		\begin{align}\label{eq:convex_rate}
			f(y(t)) - f(\bar{y}) \le  \frac{1}{t}D_{\varphi}(\bar{y},y_0).
		\end{align}
		\item If $f$ is $\mu$-strongly convex w.r.t.\ $\varphi$ then
		\begin{align}\label{eq:strongly_convex_rate}
			D_{\varphi}(\bar{y},y(t)) \le D_{\varphi}(\bar{y},y_0)e^{-\mu t},
		\end{align}
		and in that case the system \eqref{eq:mirror_flow} is exponentially stable.
	\end{enumerate}
\end{theorem}
\begin{proof} Let $u(t) = -\nabla f(\nabla \varphi^*(x(t)))$ be the state feedback law that achieves $V(x_0)$. Then
	\begin{align*}
		q(x(t),u(t)) &= - \dot{V}(x(t),u(t)) \\
		&= \langle \nabla f(y(t)), y(t) - \bar{y} \rangle \\
		&= f(y(t)) - f(\bar{y}) + D_f(\bar{y},y(t)).
	\end{align*}
	where the first equality is by Lemma~\ref{lm:Vdot_inequality} and the last equality follows by rearranging and using the definition $D_f(y,y') = f(y) - f(y') - \langle f(y'), y-y' \rangle$. Therefore, using Theorem~\ref{thm:MD_optimality} and the fact that $D_f(\cdot,\cdot) \ge 0$, we have
	\begin{align*}
		V(x_0) &\ge \int^t_0 q(x(s),u(s)) \dif s \\
		&= \int^t_0 \{ f(y(s)) - f(\bar{y})\} \dif s \\
		&\ge t \big(f(y(t)) - f(\bar{y})\big),
	\end{align*}
	where the last inequality follows from the fact that the value of the objective $f$ decreases along the output trajectory $y(t)$:
	\begin{align*}
		\frac{\dif}{\dif t}f(y(t)) &= \langle \nabla f(y(t)), \dot{y}(t)\rangle \\
		&= \langle \nabla f(y(t)), \nabla^2 \varphi^*(x(t)) \dot{x}(t)\rangle \\
		&= - \langle \nabla f(y(t)), \nabla^2 \varphi^*(x(t)) \nabla f(y(t)) \rangle \le 0
	\end{align*}
---	since $\varphi^*$ is $C^2$ and strictly convex, its Hessian $\nabla^2 \varphi^*(x)$ is positive definite for all $x \in \Reals^n$. Dividing by $t$ and using the fact that $V(x_0) = D_{\varphi^*}(x_0,\bar{x}) = D_\varphi(\bar{y},y_0)$, we get \eqref{eq:convex_rate}.
	
	When $f$ is $\mu$-strongly convex, we have
	\begin{align*}
		\frac{\dif}{\dif t}V(x(t)) &= \dot{V}(x(t),u(t)) \\
		& = \langle \nabla f(y(t)), \bar{y} - y(t) \rangle \\
		& \le f(\bar{y}) - f(y(t)) - \mu D_{\varphi}(\bar{y},y(t)) \\
		& = f(\bar{y}) - f(y(t)) - \mu D_{\varphi^*}(x(t),\bar{x}) \\
		& = f(\bar{y}) - f(y(t)) - \mu V(x(t)).
	\end{align*}
	Integrating gives the estimate
	\begin{align*}
		V(x(t)) \le e^{-\mu t}V(x_0) + \int^t_0 e^{-\mu(t-s)} \{f(\bar{y}) - f(y(s))\} \dif s
	\end{align*}
	which yields \eqref{eq:strongly_convex_rate} since $f(\bar{y}) \le f(y)$ for all $y$.
\end{proof}

\subsection{A simple example}

As a simple illustration, consider the quadratic objective $f(x) = \frac{1}{2}|Ax-b|^2$ with $A \in \Reals^{p \times n}$ and $b \in \Reals^p$ and  the quadratic potential $\varphi(x) = \frac{1}{2}|x|^2$. Assume that $A^\trn A$ is nonsingular. Then the instantaneous cost $q(x,u)$ in \eqref{eq:BE_cost} takes the form
\begin{align*}
	q(x,u) = \frac{1}{2}|Ax-b|^2 + \frac{1}{2}\langle u, (A^\trn A)^{-1}u \rangle - \frac{1}{2}|A\bar{y}-b|^2,
\end{align*}
where $\bar{y} = (A^\trn A)^{-1}A^\trn b$ is the unique minimizer of $f$. Thus, the control-theoretic interpretation of the gradient flow for this problem naturally leads to infinite-horizon optimal stabilization of a linear system with a quadratic cost.

\section{The stochastic problem}

We now consider a stochastic version of continuous-time mirror descent, dubbed \textit{mirror Langevin dynamics}, or MLD \cite{Hsieh_MLD,Chewi_MLD,Li_MLD}. The MLD generates a pair of random trajectories $(X_t,Y_t)_{t \ge 0}$ according to \eqref{eq:MLD}. The words ``Langevin dynamics'' allude to the fact that, with the quadratic potential $\varphi(x) = \frac{1}{2}|x|^2$, \eqref{eq:MLD} reduces to the usual Langevin dynamics
\begin{align*}
	\dif X_t = - \nabla f(X_t) \dif t + \sqrt{2\eps} \dif W_t.
\end{align*}
The use of MLD is mainly in the context of sampling, where one makes use of the fact that the steady-state probability density of $Y_t$ is proportional to $e^{-f/\eps}$. Moreover, since this limiting density concentrates on the set of global minimizers of $f$ as $\eps \downarrow 0$, the sampling problem is intimately related to the problem of minimizing $f$.

\subsection{Some preliminaries}

In addition to the conditions imposed on $f$ and $\varphi$ in Sec.~\ref{ssec:prelims}, we also assume the following:
\begin{itemize}
	\item the objective function $f$ has a Lipschitz-continuous gradient;
	\item the potential function $\varphi$ is $C^2$, strongly convex, cf.~\eqref{eq:phi_SC}, and has the \textit{modified self-concordance property} \cite{Li_MLD}, i.e., there exists some constant $c > 0$, such that
	\begin{align*}
		\| \sqrt{\nabla^2 \varphi(x)} - \sqrt{\nabla^2 \varphi(x')} \|_2 \le c |x-x'|
	\end{align*}
	for all $x,x' \in \Reals^n$, where $\| \cdot \|_2$ is the $2$-Schatten (or Hilbert--Schmidt) norm.
\end{itemize}
In particular, the above assumption on $\varphi$ implies that $\varphi^*$ has a Lipschitz-continuous gradient \cite[Thm.~4.2.1]{Hiriart_convex} and that the map $x \mapsto \sqrt{(\nabla^2 \varphi^*(x))^{-1}}$ is Lipschitz-continuous \cite{Li_MLD}. 

\subsection{A finite-horizon optimal control problem}

We work in the usual setting of controlled diffusion processes \cite[{\S}VI.3-4]{Fleming_Rishel_book}. Let $(\Omega, \cF, (\cF_t)_{t \ge 0},\Pr)$ be a probability space with a complete and right-continuous filtration, and let $(W_t)_{t \ge 0}$ be a standard $n$-dimensional $(\cF_t)$-Brownian motion. Let a finite horizon $0 < T < \infty$ be given. An \textit{admissible control} (of state feedback type) is any measurable function $u : \Reals^n \times [0,T] \to \Reals^n$, such that the It\^o SDE
\begin{align}\label{eq:controlled_SDE}
	\dif X_t = u(X_t,t) \dif t + \sqrt{2 \eps (\nabla^2 \varphi^*(X_t))^{-1}} \dif W_t
\end{align}
has a unique strong solution for all $t \in [0,T]$ and for any deterministic initial condition $X_0 = x_0$ (cf.~\cite[\S5.2]{Karatzas_Shreve_book} for details). For each $t \in [0,T]$ define the \textit{expected cost-to-go}
\begin{align}\label{eq:stochastic_cost_to_go}
	& J(x,t;u(\cdot))\nonumber \\
	&  \deq \Ex\Bigg[\int^T_t q(X_s,u_s)\dif s + D_{\varphi^*}(X_T,\bar{x})\Bigg|X_t = x \Bigg],
\end{align}
with the instantaneous cost $q$ the same as in \eqref{eq:BE_cost}, where $u_t$ is shorthand for $u(X_t,t)$, and let
\begin{align}\label{eq:stochastic_valfun}
	V(x,t) \deq \inf_{u(\cdot) \text{ admissible}} J(x,t; u(\cdot))
\end{align}
be the value function. We say that an admissible control $u(\cdot)$ is optimal if $J(x,t; u(\cdot)) = V(x,t)$ for all $x \in \Reals^n$ and all $t \in [0,T]$. Observe that, in contrast with the deterministic infinite-horizon problem posed in Sec.~\ref{ssec:infinite_horizon}, here we are dealing with a finite-horizon stochastic problem, and there is, in addition to the instantaneous cost $q$, also a terminal cost $D_{\varphi^*}(\cdot,\bar{x})$. The form of the performance criterion in \eqref{eq:stochastic_cost_to_go} is reminiscent of the Brezis--Ekeland action functional \eqref{eq:BE_action}.

\subsection{The main result}

\begin{theorem}\label{thm:MLD_optimality} The value function in \eqref{eq:stochastic_valfun} is equal to
	\begin{align}\label{eq:stochastic_valfun_formula}
		V(x,t) = D_{\varphi^*}(x,\bar{x}) + \eps n (T-t),
	\end{align}
	and the feedback control $u(x,t) = - \nabla f (\nabla \varphi(x))$ is optimal.
\end{theorem}

\begin{proof} We use the verification theorem from the theory of controlled diffusions \cite[{\S}VI.4]{Fleming_Rishel_book}. We associate to the controlled diffusion process \eqref{eq:controlled_SDE} a family of infinitesimal generators $(\cL^u : u \in \Reals^n)$, where $\cL^u$ is the second-order linear differential operator 
	\begin{align*}
		\cL^u \deq \sum^n_{i=1}u_i \frac{\partial}{\partial x_i} + \eps \sum^n_{i,j=1} \big(\nabla^2 \varphi^*(x)\big)^{-1}_{ij} \frac{\partial^2}{\partial x_i \partial x_j}.
	\end{align*}
Then it is readily verified that the function $V$ defined in \eqref{eq:stochastic_valfun_formula} is a solution of the \textit{Hamilton--Jacobi--Bellman equation}
\begin{align}\label{eq:HJB}
	\frac{\partial}{\partial t}V(x,t) + \min_{u \in \Reals^d} \big\{ \cL^u V(x,t) + q(x,u) \big\} = 0 
\end{align}
on $\Reals^n \times [0,T]$ with the terminal condition $V(x,T) = D_{\varphi^*}(x,\bar{x})$. Indeed, since
\begin{align*}
	\frac{\partial}{\partial t}V(x,t) &= - \eps n , \\
	\nabla V(x,t) &= \nabla \varphi^*(x) - \nabla \varphi^*(\bar{x}), \\
	\nabla^2 V(x,t) &= \nabla^2 \varphi^*(x)
\end{align*}
we can follow the same argument as in the proof of Lemma~\ref{lm:Vdot_inequality} to show that, for any $u \in \Reals^n$, we have
\begin{align*}
	& \frac{\partial}{\partial t}V(x,t) +  \cL^u V(x,t) + q(x,u) \nonumber\\
	& = - \eps n +  \langle u, \nabla V(x,t) \rangle + \eps\, {\rm tr} \big\{ (\nabla^2 \varphi^*(x))^{-1} \nabla^2 V(x,t) \big\} \nonumber \\
	& \qquad \qquad  + f(\nabla \varphi^*(x)) + f^*(-u) + \langle u, \nabla \varphi^*(\bar{x}) \rangle \\
		& = f(\nabla \varphi^*(x)) + f^*(-u) + \langle u, \nabla \varphi^*(x) \rangle  \ge 0,
\end{align*}
with equality iff $u = -\nabla f (\nabla \varphi^*(x))$. Thus, $V(x,t)$ is a solution of the HJB equation \eqref{eq:HJB}, and evidently $V(x,T) = D_{\varphi^*}(x,\bar{x})$. Then, by Theorem~4.1 in \cite[{\S}VI.4]{Fleming_Rishel_book}, $V(x,t)$ is the value function in \eqref{eq:stochastic_valfun}, and the control given by
\begin{align*}
	u(x,t) &= \argmin_{u \in \Reals^n} \big\{ \cL^u V(x,t) + q(x,u) \big\} \\
	&= \argmin_{u \in \Reals^n} \big\{  f(\nabla \varphi^*(x)) + f^*(-u) + \langle u, \nabla \varphi^*(x)  \big\} \\
	&= -\nabla f(\nabla \varphi^*(x))
\end{align*}
is optimal (note that it is also time-invariant). This control is admissible since, by our assumptions on $f$ and $\varphi$, the maps $b(x) \deq -\nabla f (\nabla \varphi^*(x))$ and $\sigma(x) \deq \sqrt{2\eps (\nabla^2 \varphi^*(x))^{-1}}$ are Lipschitz-continuous and have at most linear growth, i.e., there exists a constant $K > 0$, such that
\begin{align*}
&	|b(x)-b(x')| + \| \sigma(x) - \sigma(x') \|_2 \le K|x-x'|, \\
& |b(x)|^2 + \|\sigma(x)\|^2_2 \le K(1+|x|^2)
\end{align*}
for all $x,x' \in \Reals^n$. Consequently, with the choice of $u(x,t) = b(x)$, the SDE \eqref{eq:controlled_SDE} has a unique strong solution \cite[\S5.2, Thm.~2.5]{Karatzas_Shreve_book}, so $u(\cdot)$ is indeed admissible. 
\end{proof}

\subsection{Quantitative estimates}

\begin{theorem}\label{thm:MLD_quantitative} Let $(X_t,Y_t)$, $t \ge 0$, be the random state and output trajectories of the mirror Langevin dynamics \eqref{eq:MLD} with deterministic initial condition $X_0 = x_0$ and $Y_0 = y_0 = \nabla \varphi^*(x_0)$. Then the following holds for every $T > 0$:
	\begin{enumerate}
		\item If $f$ is convex, then
		\begin{align}\label{eq:stoch_convex_rate}
			&\frac{1}{T}\Ex\Bigg[\int^T_0 \{f(Y_t) - f(\bar{y})\} \dif t\Bigg|Y_0 = y_0\Bigg] \nonumber\\
			& \qquad \le  \frac{1}{T}D_{\varphi}(\bar{y},y_0) + \eps n.
		\end{align}
		\item If $f$ is $\mu$-strongly convex w.r.t.\ $\varphi$, for $\mu > 0$, then
		\begin{align}\label{eq:stoch_strongly_convex_rate}
			&\Ex[D_{\varphi}(\bar{y},Y_t)|Y_0 = y_0] \nonumber\\
			& \qquad \le D_{\varphi}(\bar{y},y_0)e^{-\mu t} + \frac{\eps n}{\mu}(1-e^{-\mu T}).
		\end{align}
	\end{enumerate}
\end{theorem}
\begin{proof} Let $u_t \deq -\nabla f(\nabla \varphi^*(X_t))$. Then, proceeding just like in the proof of Theorem~\ref{thm:MD_quantitative}, we can write
	\begin{align*}
		q(X_t,u_t) &= f(Y_t) + f^*(-u_t) + \langle u_t, \bar{y} \rangle \\
		&= f(Y_t) - f(\bar{y}) + D_f(\bar{y},Y_t) \\
		&\ge f(Y_t) - f(\bar{y}).
	\end{align*}
Using this together with Theorem~\ref{thm:MLD_optimality} gives
	\begin{align*}
	&	D_{\varphi}(\bar{y},y_0) + \eps n T = D_\varphi(x_0,\bar{x}) + \eps n T \\
		&= \Ex\Bigg[\int^T_0 q(X_t,u_t) \dif t + D_{\varphi}(X_T,\bar{x})\Bigg|X_0 = x_0\Bigg] \\
		&\ge \Ex\Bigg[\int^T_0 \{f(Y_t)-f(\bar{y})\}\dif t \Bigg|X_0 = x_0 \Bigg].
	\end{align*}
	Dividing both sides by $T > 0$ and using the fact the $\sigma$-algebras $\sigma(X_t : t \in [0,T])$ and $\sigma(Y_t: t \in [0,T])$ coincide since $\nabla \varphi^*$ is a bijection,  we obtain \eqref{eq:stoch_convex_rate}.
	
When $f$ is $\mu$-strongly convex, we have
\begin{align}\label{eq:stoch_SC_1}
	f(\bar{y}) - f(Y_t) \ge \langle \nabla f (Y_t), \bar{y} - Y_t \rangle + \mu D_\varphi(\bar{y},Y_t).
\end{align}
On the other hand, by It\^o's lemma and by \eqref{eq:HJB},
\begin{align}\label{eq:stoch_SC_2}
\!\!\!\!V(X_t,t) &= V(X_0,0) + \int^t_0 \langle \nabla f(Y_s), \bar{y} - Y_s \rangle \dif s + M_t,
\end{align}
where $M_t$ is a zero-mean $(\cF_t)$-martingale. Since
\begin{align*}
	V(X_s,s) &= D_{\varphi*}(X_s,\bar{x}) + \eps n (T-s) \\
	&= D_\varphi(\bar{y},Y_s) + \eps n(T-s),
\end{align*}
combining \eqref{eq:stoch_SC_1} and \eqref{eq:stoch_SC_2} and then taking expectations given $Y_0 = y_0$ yields 
\begin{align*}
	&\Ex[D_\varphi(\bar{y},Y_t)|Y_0 = y_0] \nonumber\\
	&\le D_\varphi(\bar{y},y_0) + \eps n t - \mu \int^t_0 \Ex[D_\varphi(\bar{y},Y_s)|Y_0 = y_0] \dif s
\end{align*}
for all $t \in [0,T]$. Gr\"onwall's inequality gives \eqref{eq:stoch_strongly_convex_rate}.
\end{proof}
\noindent Note that, in contrast with the deterministic setting (cf.~Theorem~\ref{thm:MD_quantitative}), when the objective function $f$ is not strongly convex, we only have guarantees on the expected average objective $\Ex[\frac{1}{T}\int^T_0 f(Y_t) \dif t | Y_0 = y_0]$, which, owing to the convexity of $f$, translates into an optimization error estimate for the time average of the trajectory, $\tilde{Y}_T \deq \frac{1}{T}\int^T_0 Y_t \dif t$:
\begin{align*}
	&\Ex[f(\tilde{Y}_T) - f(\bar{y})|Y_0 = y_0] \nonumber\\
	&\le \frac{1}{T}\Ex\Bigg[\int^T_0 \{f(Y_t)-f(\bar{y})\}\dif y\Bigg| Y_0 = y_0 \Bigg] \\
	&\le \frac{1}{T}D_\varphi(\bar{y},y_0) + \eps n.
\end{align*}
However, as the following result shows, in the low-noise regime (i.e., for all sufficiently small $\eps$), with high probability, the MLD output trajectory $(Y_t)_{0 \le t \le T}$ closely tracks the deterministic mirror-descent output trajectory $(y(t))_{0 \le t \le T}$ with the same initial condition  $Y_0 = y(0)  = y_0$:
\begin{theorem}\label{thm:MLD_low_noise} There exist positive time-independent constants $C_i$, $i = 1,2,3$, such that, for every $0 < \eps \le \frac{1}{C_1 T^3}e^{-C_2 T}$, the following estimate holds with probability at least $1-\delta$:
	\begin{align}
		\sup_{0 \le t \le T}|f(Y_t)-f(y(t))| \le \frac{C_3}{T}\sqrt{n \log \frac{n}{\delta}}.
	\end{align}
\end{theorem}
\begin{proof} Let $\Delta_t \deq |Y_t - y(t)|$. The following estimate holds by the Lipschitz continuity of $\nabla f$:
	\begin{align*}
		f(Y_t) - f(y(t)) \le \langle \nabla f(y(t)), Y_t - y(t) \rangle + \frac{L_f}{2}|Y_t - y(t)|^2,
	\end{align*}
	where $L_f$ is the Lipschitz constant of $\nabla f$. Moreover, the gradient norms $|\nabla f(y(t))|$ are uniformly bounded since
	\begin{align*}
		|\nabla f(y(t))| &\le |\nabla f(y(t))-\nabla f(y(0))| + |\nabla f(y(0))| \\
		&\le L_f |y(t)-y(0)| + |\nabla f(y(0))| \\
		& \le L_f |y(t)-\bar{y}| + L_f |y(0) - \bar{y}| + |\nabla f(y(0))| \\
		&\le 2L_f \sqrt{\frac{2}{\alpha} D_{\varphi}(\bar{y},y(0))} + |\nabla f(y(0))| =: K_0,
	\end{align*}
	which in turn implies that
	\begin{align}\label{eq:smoothness_estimate}
		\sup_{0 \le t \le T}|f(y(t))-f(y(0))| \le K_0 \sup_{0 \le t \le T}\Delta_t + \frac{L_f}{2}\sup_{0 \le t \le T}\Delta^2_t.
	\end{align}
	Define the matrix-valued process $(\xi_t)_{0 \le t \le T}$ by $\xi_t \deq \sqrt{\nabla^2 \varphi^*(X_t))^{-1}}$. For each $t \in [0,T]$, we have
	\begin{align*}
	&	|X_t - x(t)| \nonumber\\
	& \le L_f \int^t_0 |Y_s - y(s)|\dif s 
	+ \sqrt{2\eps}\sup_{0 \le t \le T}\Bigg|\int^t_0 \xi_s\dif W_s\Bigg|.
	\end{align*}
By our assumptions on $\varphi$, there exist positive constants $\kappa_2 \ge \kappa_1 > 0$, such that the eigenvalues of $\nabla^2 \varphi^*(x)$ lie in the interval $[\kappa_1,\kappa_2]$. Hence, the process $\xi_t$ is uniformly bounded, so the quadratic variations of the matrix entries $[\xi^{ij}]_t$, $1 \le i,j \le n$, are uniformly bounded by a positive multiple of $t$. Hence, by the time-change theorem for martingales \cite[\S3.4, Thm.~4.6]{Karatzas_Shreve_book}, there exist a constant $\kappa > 0$ and a standard $n$-dimensional Brownian motion $(B_t)_{t \ge 0}$, such that
\begin{align*}
	\sup_{0 \le t \le T} \Bigg|\int^t_0 \xi_s \dif W_s \Bigg| \le \sup_{0 \le t \le \kappa T} |B_t|.
\end{align*}
Since $\nabla \varphi^*$ is Lipschitz-continuous, we have
\begin{align*}
	\Delta_t &\le L_{\varphi^*} |X_t - x(t)| \nonumber \\
	&\le L_{\varphi^*} L_f \int^t_0 \Delta_s \dif s + \sqrt{2\eps} L_{\varphi^*}\sup_{0 \le t \le \kappa T} |B_t|.
\end{align*}
Gr\"onwall's inequality therefore gives
\begin{align}\label{eq:Delta_t}
	\sup_{0 \le t \le T} \Delta_t \le \sqrt{2\eps} L_{\varphi^*} \sup_{0 \le t \le \kappa T} |B_t| e^{L_{\varphi^*}L_fT}
\end{align}
If $\eps \le \frac{1}{L^2_{\varphi^*}T^3}e^{-2L_{\varphi^*}L_f T}$, then, using \eqref{eq:Delta_t}  in \eqref{eq:smoothness_estimate}, we obtain
\begin{align*}
&	\sup_{0 \le t \le T} |f(Y_t)-f(y(t))| \nonumber\\
& \qquad\le \frac{\sqrt{2}K_0}{T^{3/2}} \sup_{0 \le t \le \kappa T} |B_t| 
+ \frac{L_f}{T^3} \sup_{0 \le t \le \kappa T}|B_t|^2.
\end{align*}
By the reflection principle for the Brownian motion \cite[p.~96]{Karatzas_Shreve_book}, for every $r > 0$
\begin{align*}
	\Pr\left\{\sup_{0 \le t \le \kappa T}|B_t| \ge r\right\} \le 2\Pr\left\{ |B_{\kappa T}| \ge r \right\} \le  4n e^{-r^2/2n\kappa T},
\end{align*}
and therefore
\begin{align*}
	& \sup_{0 \le t \le T} |f(Y_t)-f(y(t))| \le \frac{\tilde{C}}{T}\sqrt{n \log \frac{n}{\delta}}
\end{align*}
with probability at least $1-\delta$, where $\tilde{C}$ is a constant that depends on $K_0,L_f,\kappa$.
\end{proof}

\section{Conclusion and future directions}

In this paper, we have presented an interpretation of deterministic and stochastic continuous-time mirror descent methods in the framework of ``inverse optimal control'' \cite{Casti_inverse_OC}---that is, given an autonomous (i.e., control-free) dynamical system, identify a controlled dynamical system and a cost criterion, such that the autonomous dynamics can be viewed as the closed-loop system corresponding to an optimal control. An intriguing direction for future research is to interpret other optimization methods, such as the heavy-ball method \cite{Polyak_heavy_ball}, through the inverse optimal control lens. 

\section*{Acknowledgments}

The authors would like to thank Jelena Diakonikolas for pointing them to Ref.~\cite{Ghoussoub_Tzou_gradflows}, which was the initial inspiration for this work, and also Anatoli Juditsky, Philippe Rigollet, and Matus Telgarsky for insightful comments and suggestions.

\bibliographystyle{IEEEtran}
\bibliography{mirror_descent.bbl}

\end{document}